\documentclass[journal, onecolumn,12pt, peerreviewca]{IEEEtran}

\usepackage{amsmath} 
\usepackage{amssymb}
\usepackage{amsthm}
\usepackage{xcolor}

\usepackage{mathrsfs}

\newcommand{\tr}{\mbox{Tr}}

\newcommand{\hess}{\textnormal{Hess}}
\newcommand{\pr}{\mbox{Pr}}

\newtheorem{proposition}{Proposition}
\newtheorem{lemma}{Lemma}
\newtheorem{theorem}{Theorem}
\newtheorem{corollary}{Corollary}
\newtheorem{remark}{Remark}

\usepackage{cite}

\begin{document}
\title{On the minmax regret for statistical manifolds: the role of curvature}

\author{
    \IEEEauthorblockN{Bruno~Mera\IEEEauthorrefmark{1}\IEEEauthorrefmark{2},~Paulo~Mateus\IEEEauthorrefmark{1}\IEEEauthorrefmark{2},~Alexandra~M.~Carvalho\IEEEauthorrefmark{1}\IEEEauthorrefmark{3}}\\
    \IEEEauthorblockA{\IEEEauthorrefmark{1}Instituto de Telecomunica\c c\~oes, 1049-001 Lisboa, Portugal
    \\\{bruno.mera, paulo.mateus, alexandra.carvalho\}@lx.it.pt}\\
    \IEEEauthorblockA{\IEEEauthorrefmark{2}Departamento de Matem\'atica, Instituto Superior T\'ecnico, Universidade de Lisboa}\\
    \IEEEauthorblockA{\IEEEauthorrefmark{3}Departamento de Engenharia Eletrot\'ecnica e Computadores,\\ Instituto Superior T\'ecnico, Universidade de Lisboa}
}

\maketitle

\begin{abstract}
Model complexity plays an essential role in its selection, namely, by choosing a model that fits the data and is also succinct. Two-part codes and the minimum description length have been successful in delivering procedures to single out the best models, avoiding overfitting. In this work, we pursue this approach and complement it by performing further assumptions in the parameter space. Concretely, we assume that the parameter space is a smooth manifold, and by using tools of Riemannian geometry, we derive a sharper expression than the standard one given by the stochastic complexity, where the scalar curvature of the Fisher information metric plays a dominant role. Furthermore, we derive the minmax regret for general statistical manifolds and apply our results to derive optimal dimensional reduction in the context of principal component analysis.
\end{abstract}


\section{Introduction}
\IEEEPARstart{T}{wo}-part codes are an essential tool in model selection. Not only they optimize the likelihood of the data given the model, but they also take into account model complexity.
There has been a line of research where one considers, in the most abstract setting, families of distributions satisfying minimal requirements and derives an expression for model complexity, such as the stochastic complexity, among others~\cite{ris:96,suz:yam:18}. These formulas are sharp to the extent of the absence of assumptions in the assignment of a probability distribution to each point in the parameter space. Moreover, it is a rather usual assumption that this parameter space has the topology of an open subset in $\mathbb{R}^n$. 

 In this paper, we show that by making additional assumptions on the parameter space and endowing it with natural information geometric structures,  we can arrive to sharper results by applying techniques from Riemannian geometry.   In practice, parameters of the distributions are usually taken to live on a smooth manifold, and the distribution is assumed to vary smoothly with the parameters.
 However, usually one takes the simplification that this manifold is a trivial open subset of the Euclidean space. 
 In this work, we will drop this assumption, hence allowing for non-trivial topologies.
   Moreover, Information Theory endows the manifold with a positive (semi-)definite covariant $2$-tensor, namely a Riemannian metric -- the Fisher information~\cite{ama:07, ama:12}. 
   Since we are given a Riemannian structure, we have a natural notion of a uniform distribution over the manifold of parameters, which corresponds to what is known in the literature as Jeffreys' prior~\cite{jeff:46, cla:94}. 
   
 In the literature, when the parameter space is just a bounded open set in $\mathbb{R}^n$,
  one can find the (normalized) maximum likelihood code, defined by
   \begin{equation}\label{eq:nml}
   	p^*(x^N)=\frac{p(x^N|\hat \theta)}{\int_{y^N\in \mathcal{X}^N}p(y^N|\hat \theta) d y^N}.
   \end{equation} 
   The associated length was firstly given by Rissanen~\cite{ris:96}, computed through Laplace's formula, and has the form
   \begin{equation}\label{eq:ris}
   	L^*(x^N)=-\log(p^*(x^N))=-\log p(x^N|\hat{\theta})+\frac{n}{2}\log\left(\frac{N}{2\pi}\right)+
   	\log \int \sqrt{|I(\theta)|} d\theta +o(1),
   \end{equation} 
   where the expansion is stated in terms of the size of the dataset $N$.
   While in Rissanen's original work he considered $x^N$ beyond i.i.d. processes, in the present work we will only focus in this case. Observe that Eq.~\eqref{eq:ris} does not account for the possible dependence of the $o(1)$ term in the dimension of the parameter space. Indeed, in this work, using techniques from Riemannian Geometry, we find the sharper formula 
   
      \begin{align}\label{eq:geomcomp}
L^*(x^N)&=- \log p(x^N|\hat{\theta})+\frac{n}{2}\log\left(\frac{N}{2\pi}\right)+\log \textnormal{vol}_g(M)\\
&
\underbrace{-\log\left(\frac{\sqrt{\det(g_{\hat{\theta}})}}{\sqrt{\det(I(x^N,\hat{\theta}))}}\right)-\frac{1}{6N}R(\hat{\theta}) +\textnormal{O}\left(\frac{1}{N^2}\right)}_{o(1)\text{ as a function of }N},\nonumber
\end{align}  
where three classical geometric invariants can be easily identified, namely: (i) the dimension of the manifold $n$, (ii) the Riemannian volume $\textnormal{vol}_g(M)$; and (iii) the Ricci scalar curvature $R(\hat{\theta})$ evaluated at the maximum likelihood estimate $\hat\theta$. While in Eq.\eqref{eq:ris} the term $\log \int \sqrt{|I(\theta)|} d\theta$ is precisely the logarithm of the Riemannian volume, we choose to write it explicitly to highlight its geometric nature. Note that the scalar curvature might be very large as a function of the type of data involved. For example, currently it is very common to have high dimensional data and this curvature will most likely depend on this dimension, as it is the case of Gaussian models, as we shall see below.

To derive Eq.~\eqref{eq:geomcomp}, motivated by the results in~\cite{bal:05},  we follow a Bayesian approach considering Jeffreys' prior and we adapt Laplace's method to manifolds, using canonical Riemann normal coordinates to our advantage.

In order to obtain the minmax regret akin to Eq.~\eqref{eq:nml}, we use Haussler's version of the capacity theorem~\cite{haus:97} that requires the map $p:\theta\mapsto p(\cdot|\theta)$ to be continuous with respect with the weak topology on the target space of probability distributions on $\mathcal{X}^N$, that is, for every bounded continuous function $f$ we have that 
\begin{equation}\label{eq:weakconv}
	\text{ if }\theta_n\to \theta \text{ then }E_{p(\cdot|\theta_n)}[f]\to E_{p(\cdot|\theta)}[f],
\end{equation}
where $(\theta_n)_{n\in \mathbb{N}}$ is a (convergent) sequence in $M$.
 In~\cite{cla:bar:90}, such condition is present and equivalent to the {\em soundness assumption} of the parametrization. Since locally, in a smooth manifold, everything looks like an open set in $\mathbb{R}^n$, the natural condition to take is that such soundness holds for every coordinate neighborhood, property that we call {\em local soundness assumption} of the statistical model. Under this assumption, we show that the minmax regret of data $x^N$ generated by $\theta_0$ is given by  
\begin{equation}\label{eq:regret}
\mathcal{R}_{N}(x^N)=\frac{n}{2}\log\left(\frac{N}{2\pi}\right) +\log \textnormal{vol}_g(M)-\log\left(\frac{\sqrt{\det(g_{\theta_0})}}{\sqrt{\det(I(x^N,\theta_0))}}\right)-\frac{1}{6N}R(\theta_0) +\textnormal{O}\left(\frac{1}{N^2}\right).
\end{equation}
Observe Eq.~\eqref{eq:geomcomp} follows from this result by adding the length of the optimal code, $-\log p(x^N|\theta_0)$, and replacing $\theta_0$ with the unique (by assumption) estimator $\hat\theta$ in the manifold. Thus, we can see Eq.~\eqref{eq:geomcomp} as a two-part code, where Eq.~\eqref{eq:regret}, with $\theta_0$ replaced by $\hat\theta$, is a refinement of the stochastic complexity~\cite{ris:96}, taking into account the geometry of the statistical model, and therefore we call it {\em Geometric Complexity}.

 We apply our results to a very well established method for dimensional reduction, namely, Principal Component Analysis (PCA). In particular, our results yield a natural criterion for the choice of the optimal dimension, by adapting the two-part code given in Eq.~\eqref{eq:geomcomp} to zero mean Gaussian families with varying covariance. The underlying parameter space is the manifold $\mathscr{P}_m$ of positive definite matrices, with reduced dimension $m\times m$ which we want to optimize, equipped with the Fisher metric. We considered a bounded subset $M(s)$ of $\mathscr{P}_m$, controlled by an integer $s$ that is the smallest integer such that $I_{d}\leq \Sigma \leq 2^{2s} I_{d}$, where $\Sigma=XX^T/N$ is the empirical covariance matrix and $I_d$ is the $d\times d$ identity matrix. We also assume the each component of the data is written as an integer multiple of the precision for each variable, and therefore the volume depends on the precision and not in a particular system of units. For this particular case, the formula becomes 
 
 \begin{equation}\label{eq:gaussgc}
L^*(x^N)=- \log p(x^N|\hat{Q})+\frac{m(m+1)}{4}\log\left(\frac{N}{2\pi}\right)+\log \textnormal{vol}_g(M(s))+\frac{(m+2)m(m-1)}{24N}, 
 \end{equation}
 where 
\begin{align*}\label{eq:logvol}
\log\text{vol}_g\left(M(s)\right)=&-\frac{3}{2}m -\log(m!) +m\log(2)+\frac{m(m+1)}{4}\log(\pi)\\
&-\log \left(\frac{\pi^{1/4} A^{3/2} G\left(\frac{m}{2}-\frac{1}{4} (-1)^{m+1}+\frac{3}{4}\right)}{{2}^{1/24} e^{1/8}}\right)-\log \left(G\left(\left\lfloor \frac{m}{2}\right\rfloor +1\right)\right)
 +\log I(s),
\end{align*}
$A$ is the Glaisher constant, $G$ is the Barnes $G$-function, and
\begin{align*}
I(s)=s^{m}\left(\log(2)\right)^m 8^{\frac{m(m-1)}{4}}\int_{[0,1]^m}\prod_{1\leq i < j\leq m}\sinh\left(s\log(2)|u_i-u_j|\right)\prod_{i=1}^{m}du_i,
\end{align*} 
whose asymptotic behavior with $s$ is studied in Section~\ref{sec:pca}. Notice that the fourth term of Eq.~\eqref{eq:geomcomp} does not appear in our expression, since it is exactly zero for Gaussian models. Remarkably, the curvature term is negative due to the hyperbolic nature of the geometry of Gaussian statistical models, which brings a positive correction to $L^{*}(x^N)$. This correction is expected to be particularly relevant for high dimensional data. 

In Section~\ref{sec:rg}, we start by recalling some results on information geometry and then, by extending Laplace's method to manifolds (the proof of which is novel, and can be found in the appendix), under suitable conditions, we are able to derive an asymptotic formula for the posterior according to Jeffrey's prior. In Section~\ref{sec:minmax}, we derive the minmax regret for general statistical manifolds, assuming locally sound smooth families of probability distributions. In Section~\ref{sec:pca}, we apply our results to dimensional reduction in the context of principal component analysis. Finally, we draw some conclusions and present an outlook in Section~\ref{sec:conc}.

\section{The Riemannian geometry underlying Jeffrey's prior}\label{sec:rg}

Let $M$ be a smooth closed (compact and without boundary), connected, oriented manifold of dimension $n$, and $S=\{p(X|\theta)\}_{\theta\in M}$ a smooth family of probability distributions modeling a random variable $X$ taking values in the space of outcomes $\mathcal{X}$. By a smooth family of probability distributions we mean that the map $M\ni\theta\mapsto p(X=x|\theta):=p(x|\theta)\in \mathbb{R}$ is smooth for every $x\in\mathcal{X}$. We will also assume that the map is injective, i.e., the statistical model is said to be \emph{identifiable}. The set $S$ is also known as a statistical model or a parametric model. It is often the case that $M\subset \mathbb{R}^k$, for some $k$, but we choose to leave it as a general abstract manifold. We refer to the pair $(M,p(X|.))$ as a statistical manifold. The map $p:M\ni\theta\mapsto p(X|\theta)$ allows, by pullback, to define a (possibly degenerate) Riemannian structure on $M$ known in the literature as the Fisher-Information metric~\cite{ama:07,ama:12}:
\begin{eqnarray}
g(\theta) &=& E_{\theta}[d\log p(X|\theta)\otimes d\log p(X|\theta)]\nonumber\\
&=& \sum_{\mu,\nu=1}^{n}g_{\mu\nu}(\theta)d\theta^{\mu}d\theta^{\nu},
\end{eqnarray}
where $E_{\theta}$ denotes the expectation value with respect to the probability distribution $p(X|\theta)$ and $(\theta^{1},...,\theta^n)$ are arbitrary local coordinates on the manifold $M$. The locally defined matrix $[g_{\mu\nu}(\theta)]_{1\leq \mu,\nu\leq n}$ is usually referred to as the Fisher information matrix and it is a measure of the amount of information that an observable random variable $X$ carries about an unknown parameter $\theta$ of $p(X|\theta)$ modelling $X$.

If we have a discrete and finite space of outcomes, say $\mathcal{X}=\{1,...,N\}$, then a statistical model is described by smooth functions $\{p_i(\theta)\geq 0 : i=1,...,N\}$ with $\sum_ip_i(\theta)=1$ and	
\begin{equation}
g(\theta)=\sum_{i=1}^{N}\sum_{\mu,\nu=1}^{n}\frac{1}{p_i(\theta)}\frac{\partial p_i}{\partial \theta^{\mu}}(\theta)\frac{\partial p_i}{\partial \theta^{\nu}}(\theta)d\theta^{\mu}d\theta^{\nu}.
\end{equation}
If one considers the standard simplex $\Delta^{N-1}=\{(p_1,...,p_{N})\in\mathbb{R}^{N}:\sum_{i=1}^{N}p_i=1,\ p_i\geq 0\}$, then the map $\Phi: \Delta^{N-1}\ni (p_1,...,p_{N})\mapsto (\sqrt{p_1},...,\sqrt{p_{N}})\in S^{N-1}$, where $S^{N-1}$ denotes the unit sphere in $\mathbb{R}^{N}$, provides a homeomorphism onto the image and endows $\Delta^{N-1}$ with the structure of a smooth manifold. Furthermore, if we equip the sphere $S^{N-1}\subset \mathbb{R}^{N}$ with the standard round metric, then $\Delta^{N-1}$ canonically inherits, by restriction, the structure of a Riemannian manifold $(\Delta^{N-1},g_{\text{can}})$. The Fisher metric on $M$ is, up to a multiplicative constant factor (this constant is equal to $4$), the metric induced on $M$ by the map $p:M\ni \theta\mapsto p(X|\theta)\in \Delta^{N-1}$. 
Yet another description of the Fisher metric is provided by the formula
\begin{equation}
g_{\mu\nu}(\theta)=-E_{\theta}\left[\frac{\partial^2\log p(X|\theta)}{\partial\theta^\mu\partial \theta^\nu}\right], \text{ with } \mu,\nu=1,...,n.
\end{equation}
Among the various important features of this metric is its role in the Cram\'{e}r-Rao inequality theorem~\cite{ama:07}, which states that the covariance matrix of an unbiased estimator minus the inverse of the Fisher information matrix is positive semi-definite. As a consequence, the Fisher information provides the covariance of the best unbiased estimator, in the sense that its variance is the minimum possible. 

Suppose we are are given a collection of i.i.d observations of the random variable $X$, $x^{N}=(x_1,...,x_N)$. We wish to infer the best statistical model describing the data set $x^{N}$. Given a statistical model $S=\{p(X|\theta)\}_{\theta\in M}$, the probability distribution governing $x^N\in \mathcal{X}^{N}$ is given by
\begin{align}
p(x^N|\theta)=\prod_{i=1}^{N}p(x_i|\theta).
\end{align}
We may then take the random vector $X^N$ taking values in $\mathcal{X}^N$ corresponding to the $N$ observations of the single random variable $X$ and describe it through the statistical model $S_N=\{p(X^N|\theta)\}_{\theta\in M}$ such that $p(X^N=x^N|\theta)=p(x^N|\theta)$. If we denote by $g(\theta)$ and $g_N(\theta)$ the Fisher metrics associated with $S$ and $S_N$, respectively, we have:
\begin{align}
g_{N}(\theta)=Ng(\theta).
\label{eq: extensive property of g}
\end{align}
We shall refer to Eq.~\eqref{eq: extensive property of g} as the extensive property of the Fisher metric.
As a consequence, the geometry of $S$ and that of $S_N$ are the same modulo the scale factor $N$. 

In the absence of additional information, the Fisher metric allows us to introduce a probability distribution on $M$. This probability distribution has the interpretation of a uniform probability distribution for the statistical model $S_N$ and it is called Jeffreys' prior in the field of Bayesian statistics. The associated probability density is given by the top differential form
\begin{align*}
\frac{\sqrt{\det[g_N(\theta)]}d\theta^1\wedge...\wedge d\theta^n}{\text{vol}_{g_N}(M)},
\end{align*}
where the normalization factor is the Riemannian volume of $M$ according to the Fisher metric $g_N$:
\begin{align*}
\text{vol}_{g_N}(M)=\int_{M}\sqrt{\det[g_N(\theta)]}d\theta^1\wedge...\wedge d\theta^n.
\end{align*}
Notice that if $M$ is compact, this integral is very well defined, but if $M$ is not compact one has to regularize this integral in some way. By the extensive property of the Fisher metric, Eq.~\eqref{eq: extensive property of g}, this probability distribution is the same as the one provided by $g$:
\begin{align*}
\frac{\sqrt{\det[g_N(\theta)]}d\theta^1\wedge...\wedge d\theta^n}{\text{vol}_{g_N}(M)}=\frac{\sqrt{\det[g(\theta)]}d\theta^1\wedge...\wedge d\theta^n}{\text{vol}_{g}(M)}.
\end{align*}
From now on, for the sake of simplicity, we will denote by $dV_{g}:=\sqrt{\det[g(\theta)]}d\theta^1\wedge...\wedge d\theta^n$. 

In a Bayesian perspective, the probability of the statistical model $S$ (or equivalently of $S_N$) given the observed data $x^{N}$, $\pr (S|x^{N})$, is given by
\begin{align*}
\pr(S|x^{N})=\frac{\pr(S)}{\pr(x^{N})}\times \int_{M}p(x^N|\theta)\frac{dV_g}{\text{vol}_{g}(M)} ,
\end{align*}
where $\pr(S)$ and $\pr(x^{N})$ denote the prior probabilities of the statistical model $S$ and the data $x^{N}$, and 
$\int_{M}p(x^N|\theta)dV_g/\text{vol}_{g}(M)$ is our posterior likelihood according to the prescription of Jeffreys' prior. Without prior knowledge of details of the true distribution of $X$, any statistical model $S$ should be equally likely. Maximizing $\pr(S|x^{N})$ is therefore equivalent to maximizing the functional
\begin{align*}
F(x^N,S)=\int_{M}p(x^N|\theta)\frac{dV_g}{\text{vol}_{g}(M)},
\end{align*}
with respect to the statistical model $S=\{p(X|\theta)\}_{\theta\in M}$. Mathematically, finding a maximum for $F$ is a very difficult problem since the space of all such $S$ is very complicated. Namely, we are considering the union over all smooth manifolds $M$ of the spaces of maps from these manifolds to the set of probability distributions on a given outcome space $\mathcal{X}$, namely $p:\theta\mapsto p(X|\theta)$ such that, for every $x\in\mathcal{X}$, $M\ni \theta\mapsto p(x|\theta)\in\mathbb{R}$ is smooth. However, we can go a bit further than this by using the Riemannian structure on $M$ and the assumption that $N$ is large. We re-write the functional $F(S)$ as  
\begin{align}
F(x^N,S)= \int_{M} e^{-N f(\theta)} \frac{dV_g}{\mbox{vol}_g(M)},
\end{align}
with $f(\theta):=-(1/N)\log p (x^N|\theta)$. Notice that the minima of $f$ are precisely the maximum likelihood parameters denoted by $\hat{\theta}\in M$. The minimum of $f$, in the large $N$ limit, is unique because we assume that the statistical model in identifiable. In the following, we will perform a saddle point approximation to this integral, valid in the limit when $N$ is large.
%
%

We will use the following theorem which is a generalization of Laplace's method in $\mathbb{R}^n$ for the case of closed oriented Riemannian manifolds.

\begin{theorem} (Laplace's method) Let $(M,g)$ be a Riemannian closed oriented manifold of dimension $n$, where $g$ is the Riemannian metric, let $dV_g$ denote the Riemannian volume form and $f$ a smooth function with a single maximum at $p_0\in M$. Then,
\begin{align*}
\lim_{N\to\infty}\frac{\int_{M} e^{Nf}dV_g}{\left(\frac{2\pi}{N}\right)^{n/2}e^{Nf(p_0)}\!\!\frac{\sqrt{\!\det(g_{p_0})}}{\sqrt{\det(\hess_{p_0}(f))}}\left[\!1\!+\!\frac{1}{6N}\textnormal{tr}(\hess_{p_0}(f)^{-1}R_{p_0})\!\right]}=1
\end{align*}
where $R_{p_0}$ denotes the Ricci tensor at $p_0$.
\label{th: Laplace Riemannian orientable closed}
\end{theorem}

We leave the proof to the Appendix of this paper.

\begin{corollary} (Saddle point approximation) Under the same conditions of Theorem~\ref{th: Laplace Riemannian orientable closed}, it follows that, as $N\to\infty$,
\begin{align*}
-\log\int_{M} e^{Nf}dV_g &= - Nf(p_0) +\frac{n}{2}\log\left(\frac{N}{2\pi}\right) -\log\left(\frac{\sqrt{\det(g_{p_0})}}{\sqrt{\det(\hess_{p_0}(f))}}\right)\\
&-\frac{1}{6N}\textnormal{tr}(\hess_{p_0}(f)^{-1}R_{p_0}) +\text{O}\left(\frac{1}{N^2}\right).
\end{align*}
\label{corol:1}
\end{corollary}

The strong law of large numbers, which applies to independent identically distributed random variables, ensures that the random variable $-(1/N)\log p(X^N|\theta)=-(1/N)\sum_{i=1}^{N}\log p(X_i|\theta)$ satisfies
\begin{align*}
\pr\left[\lim_{N\to\infty}\left(-\frac{1}{N}\log p(X^N|\theta)\right)=E[-\log p(X|\theta)]\right]=1,
\end{align*}
in other words, the function $f(\theta)=-(1/N)\log p (x^N|\theta)$, as $N\to\infty$, approaches the entropy of the distribution $p(X|\theta)$. Moreover, if we take local coordinates $(\theta^1,...,\theta^n)$, we can define the matrix
\begin{align*}
I(x^N,\theta)=[I_{\mu\nu}(x^N,\theta)]_{1\leq\mu,\nu\leq n}:=\left[-\frac{1}{N}\frac{\partial^2\log p(x^N|\theta)}{\partial\theta^\mu\partial \theta^\nu}\right]_{1\leq \mu,\nu\leq n}.
\end{align*}
We further have that, by smoothness and the strong law of large numbers, 
\begin{align*}
I_{\mu\nu}(x^N,\theta)\to E\left[-\frac{1}{N}\frac{\partial^2\log p(X^N|\theta)}{\partial\theta^\mu\partial \theta^\nu}\right]= g_{\mu\nu}(\theta), \text{ as } N\to\infty,
\end{align*}
for all $\mu,\nu=1,...,n$. We can then apply the results of Theorem~\ref{th: Laplace Riemannian orientable closed} to get
\begin{align*}
-\log F(x^N,S)&=- \log p(x^N|\hat{\theta})+\frac{n}{2}\log\left(\frac{N}{2\pi}\right) \\
&+ \log \text{vol}_g(M)-\log\left(\frac{\sqrt{\det(g_{\hat{\theta}})}}{\sqrt{\det(I(x^N,\hat{\theta}))}}\right)\\
&-\frac{1}{6N}\textnormal{tr}\left[\left(I(x^N,\hat{\theta})\right)^{-1}R_{\hat{\theta}}\right]+\text{O}\left(\frac{1}{N^2}\right).
\end{align*}
Furthermore, it is safe to replace $I(x^N,\hat{\theta})^{-1}$ by $g^{-1}(\hat{\theta})$, because their difference must go to zero as $N\to \infty$ and, hence, when multiplied by $-1/6N$, the result will go faster to zero than $1/N$. Thus, we get the following theorem which is one of the main results of our paper:

\begin{theorem}\label{thm:2} Let $S=\{p(X|\theta)\}_{\theta\in M}$ be a smooth statistical model for closed oriented $M$. Let $g$ denote the Fisher metric so that the pair $(M,g)$ is a Riemannian manifold. Then, the functional $-\log F(x^N,S)$ has the following large $N$ asymptotic expansion:
\begin{align*}
-\log F(x^N,S)&=- \log p(x^N|\hat{\theta})+\frac{n}{2}\log\left(\frac{N}{2\pi}\right)\\
&+\log \textnormal{vol}_g(M)-\log\left(\frac{\sqrt{\det(g_{\hat{\theta}})}}{\sqrt{\det(I(x^N,\hat{\theta}))}}\right)\\
&-\frac{1}{6N}R(\hat{\theta}) +\textnormal{O}\left(\frac{1}{N^2}\right),
\end{align*}
where $R(\hat{\theta}):=\sum_{\mu,\nu=1}^{n}g^{\mu\nu}(\hat{\theta})R_{\mu\nu}(\hat{\theta})$ denotes the Ricci scalar curvature at $\hat{\theta}$ and $[g^{\mu\nu}(\theta)]_{1\leq\mu,\nu\leq n}$ is the inverse of $[g_{\mu\nu}(\theta)]_{1\leq\mu,\nu\leq n}$. 
\end{theorem}

\section{The minmax regret for general statistical manifolds}\label{sec:minmax}

Herein we obtain the minmax regret in the present context of statistical manifolds. We begin by considering a natural assumption, which generalizes the soundness condition of Clark and Barron in Ref.~\cite{cla:bar:90}. Concretely, we assume that the smooth family $\{p_{\theta}\}_{\theta\in M}$ is {\em locally sound}, i.e., let $U\subset M$ be a coordinate neighborhood, with $\phi: U\subset M\to \phi(U)\subset \mathbb{R}^n$ the chart, then the induced map from $\phi(U)$ to the set of probability distributions with space of outcomes $\mathcal{X}$ is sound. 
 According to this definition, if $(\phi(\theta_n))$ is a sequence converging in Euclidean norm to $\phi(\theta)$, denoted by $\phi(\theta_n)\to \phi(\theta)$, then $(p_{\theta_n})$ weakly converges to $p_{\theta}$, also denoted by $p_{\theta_n}\to p_{\theta}$. Weak convergence means that for every bounded continuous function $f:\mathcal{X}\to \mathbb{R}$, we have that $E_{p_{\theta_n}}[f]\to E_{p_{\theta}}[f]$. 

The previous assumption has two important consequences. In proving the results, Clark and Barron assume that the posterior distribution is sound. That implies that the latter localizes on neighborhoods of the true value of the distribution at a fast enough rate so that they can use Laplace's approximation. In the present situation, the equivalent statement is made on
\begin{align*}
p^{N}(\theta|x^N)=\frac{w(\theta)p^N(x^N|\theta)}{m_N(x^N)},
\end{align*}
which is taken to be locally sound in the sense described above. In the previous formula, $w(\theta)dV_g$ is a top form on the manifold $M$ (notice that for Jeffreys' prior $w(\theta)=1/\text{vol}_g(M)$ is the uniform distribution with respect to the Riemannian metric), and $m_{N}(x^N)=\int_{M}w(\theta)p^N(x^N|\theta)dV_g$. As a consequence, we can apply Laplace's formula form Riemannian manifolds (see Corollary~\ref{corol:1}). Secondly, the local soundness condition implies that the Haussler's version of the capacity theorem holds, see~\cite{haus:97}. Such result states that the following two quantities (actually there is a third one that we do not use here) are equal
\begin{align*}
\sup_{w}\inf_{q} I(w, q) =\inf_{q}\sup_{\theta\in M} D_{KL}\Big(p(x^N|\theta)\mathbin{||} q(x^N)\Big)=:\mathcal{R}_N,
\end{align*}
where $I(w,q)=\int_M w(\theta)D_{KL}\Big(p(x^N|\theta)\mathbin{||} q(x^N)\Big)d V_g$ is the cross information between $M$ under $w(\theta)dV_g$ and $X$ under $q$. 

The following two technical lemmas are useful to derive the minmax regret in the present setup. 

\begin{lemma} For all distributions $q$ on the $N$-fold cartesian product $\mathcal{X}^N$, we have
\begin{align*}
\int_{M}w(\theta)D_{KL}\left(p^N_\theta||q\right)dV_g= \int_{M}w(\theta)D_{KL}\left(p^N_\theta||m_N\right)dV_g+D_{KL}(m_N||q),
\end{align*}
where $p^N_{\theta}(x^N)=\prod_{i=1}^N p(x_i|\theta)$. Hence, 
\begin{align*}
\inf_{q}\int_{M}w(\theta)D_{KL}\left(p^N_\theta||q\right)dV_g=\int_{M}w(\theta)D_{KL}\left(p^N_\theta||m_N\right)dV_g.
\end{align*}
\label{lemma:1}
\end{lemma}

The proof of the previous lemma follows easily by noticing that $m_N$ and $q$ do not depend on $\theta$ and $\int_{M}w(\theta)dV_g=1$.

\begin{lemma}
\begin{align*}
\int_{M}w(\theta)D_{KL}\left(p_{\theta}^N||m_{N}\right)dV_g=-D_{KL}\left(w||w^{\textnormal{Jeffreys}}\right)+\frac{n}{2}\log\frac{N}{2\pi} +\textnormal{o}(1).
\end{align*}
\label{lemma:2}
\end{lemma}

\begin{proof}
The local soundness assumption on $p(x|\theta)$ yields localization, at a sufficiently fast rate~\cite{cla:bar:90}, of the distribution 
\begin{align*}
p(\theta|x^N)=\frac{w(\theta)p(x^N|\theta)}{m_{N}(x^N)}
\end{align*}
on a neighborhood of $\theta_0\in M$, where $\theta_0$ is the value of $\theta$ that generates the data $x^N$. The argument for localization goes as follows. Let $\{U_{\alpha}\}_{\alpha\in A}$ be an open covering of $M$ by coordinate neighborhoods with $\phi_{\alpha}:U_{\alpha}\to\mathbb{R}^n$ the chart map. Then over $\phi(U_{\alpha})$, $\alpha\in A$, the family of distributions $\{p(\theta=\phi^{-1}(\xi)|x^N)\}_{\xi\in \phi(U_{\alpha})\subset \mathbb{R}^n}$ is sound as in the definition of Clark and Barron~\cite{cla:bar:90}. It follows by their results that the distribution localizes on $\phi_{\alpha}(\theta_0)$ for some $\alpha \in A$, i.e., an open set containing $\theta_0$, where $\theta_0$ is the value of $\theta$ that generated the data $x^N$.

This fact allows for the use of Laplace's approximation, generalized for manifolds, on the integral defining $m_N(x^N)$. Concretely, we have,
\begin{align}
m_N(x^N)&=\int_{M}w(\theta)p(x^N|\theta)dV_g \nonumber\\
&=w(\theta_0)p(x^N|\theta_0)\times \left(\frac{2\pi}{N}\right)^n\times \frac{\sqrt{\det g_{\theta_0}}}{\sqrt{\det I_{\theta_0}}}\times\left(1+\frac{1}{6N}\tr\left(I_{\theta_0}^{-1}R_{\theta_0}\right)+\frac{1}{N}c+\text{O}\left(\frac{1}{N^2}\right)\right),
\label{eq: lemma2}
\end{align}
where $c$ is a constant which depends on the Hessian of $w$ and it is $0$ for Jeffreys' prior. For the purpose of this proof, it is enough to keep the terms up to $\text{O}(1)$. The lemma follows by applying the resulting expression for $m_N$ on $\int_{M}w(\theta)D_{KL}\left(p_{\theta}^N||m_{N}\right)dV_g$.

\end{proof}

\begin{theorem}\label{thm:regret} Let $\{p(X|\theta)\}_{\theta\in M}$ be a locally sound smooth family of probability distributions over $\mathcal{X}$, where $M$ is an oriented smooth manifold of dimension $n$. Let $x^N$ be a data set generated by the probability distribution $p^{N}(X^N|\theta_0)$ for some $\theta_0\in M$. The minmax regret $\mathcal{R}_{N}(x^N)$ is given by
\begin{align*}
\mathcal{R}_{N}(x^N)=\int_{M}w^{\textnormal{Jeffreys}}(\theta)D_{KL}(p^{N}_{\theta}\mathbin{||}m^{\textnormal{Jeffreys}}_{N})dV_g&=\frac{n}{2}\log\left(\frac{N}{2\pi}\right) +\log \textnormal{vol}_g(M)-\log\left(\frac{\sqrt{\det(g_{\theta_0})}}{\sqrt{\det(I(x^N,\theta_0))}}\right)\nonumber\\
&-\frac{1}{6N}R(\theta_0) +\textnormal{O}\left(\frac{1}{N^2}\right).
\end{align*}

\begin{proof} Given the assumption that $p(x|\theta)$ is locally sound, we have the topology of weak convergence (i.e. the topology as defined by $\beta$ in Haussler's paper~\cite{haus:97}). Haussler's version of the capacity theorem gives
\begin{align*}
\sup_{w}\inf_{q} I(w, q) =\inf_{q}\sup_{\theta\in M} D_{KL}\Big(p(x^N|\theta)\mathbin{||} q(x^N)\Big)=:\mathcal{R}_N.
\end{align*}
By Lemma~\ref{lemma:1}, we conclude 
\begin{align*}
\mathcal{R}_{N}=\sup_{w}I(w,m_{N})=\sup_{w}\int_{M} w(\theta) D_{KL}\left(p^{N}_{\theta}\mathbin{||}m_{N}\right)dV_g.
\end{align*}
By Lemma~\ref{lemma:2}, it follows that the supremum is achieved for $w=w^{\textnormal{Jeffreys}}$. Finally, if in the proof of Lemma~\ref{lemma:2} we replace $w$ by $w^{\text{Jeffreys}}$ and keep all the terms as in Eq.~\eqref{eq: lemma2}, the result follows.

\end{proof}

\end{theorem}

 Observe Theorem.~\ref{thm:2} follows from this result by adding the length of the optimal code, $-\log p(x^N|\theta_0)$, and replacing $\theta_0$ with the unique (by assumption) estimator $\hat\theta$ in the manifold.

\section{Application to PCA}\label{sec:pca}
Let $x^{N}=(x_1,...,x_N)\in\mathcal{X}^{N}$ be a data set, where now we take $\mathcal{X}=\mathbb{R}^{d}$, thus $x^{N}$ will be interpreted as a $d\times N$ real-valued matrix. Suppose that the empirical mean $\bar{x}=(1/N)\sum_{i=1}^{N}x_i$ vanishes. If it does not, we can always shift the data by the empirical mean so that the transformed data satisfies this requirement. Let $\Sigma=x^{N}\left({x^{N}}\right)^T/N$ be the empirical covariance matrix and assume that $s$ is the smallest integer such that  $\Sigma \leq 2^{2s} I_{d}$, where $I_d$ is the $d\times d$ identity matrix. For the data points to be independent of a unit system, we assume all the data to be an integer multiple of the some fundamental precision. With this convention, all covariance matrices $\Sigma$ are such that $I_{d}\leq \Sigma$. Moreover, let $\Lambda= \text{Tr} (\Sigma)$, then $\Lambda \leq d \; 2^{2s}$. 
The principal component analysis (PCA) is a method for dimensional reduction of the data using the information contained in the empirical covariance $\Sigma$. Namely, given the dimension $d$ of the Euclidean space where the data points live in, we construct a new covariance matrix $\Sigma_{r}$ as follows. Let $S$ be a rotation matrix of eigenvectors of $\Sigma$, so that
\begin{align*}
\Sigma= S\text{diag}(\lambda_1,...,\lambda_d) S^t.
\end{align*}
By applying a permutation matrix if necessary, we may assume that $\lambda_i\geq \lambda_i+1$, $i=1,...,d$. The idea is to simplify the representation of the data by taking the first $m$ directions of distinguishability and simplify the description of the others by taking an isotropic subspace where the variance is the average the remaining ones. Explicitly, 
\begin{align*}
\Sigma_r=S\left(\text{diag}(\lambda_1,...,\lambda_{m})\oplus \bar{\lambda} I_{d-m}\right)S^t,
\end{align*}
where $\bar{\lambda}=(\Lambda -\sum_{i=1}^{m}\lambda_i)/(d-m)$. 

 The problem is to find a criterion to determine an optimal $m$. In the following, by using the results of the previous sections, we will provide one natural criterion. We will write
\begin{align*}
S=[v_1,...,v_d]=[A \ B],
\end{align*}
where $A=[v_{1},...,v_{m}]$ and $B=[v_{m+1},...,v_d]$, and $v_i\in \mathbb{R}^d$, $i=1,...,d$. Let $V_{A}=\text{span}\{v_{1},...,v_{m}\}$, with $\dim V=m$, be subspace generated by the first $m$ columns of $S$ and similarly for $V_{B}=\text{span}\{v_{m+1},...,v_{d}\}$. It is clear that $V_{A}\oplus V_{B}$ is an orthogonal decomposition of $\mathbb{R}^d$. We take as our statistical model the family of Gaussian distribution centered at $0\in \mathbb{R}^d$, whose covariance matrix assumes the form:
\begin{align}\label{Eq: Q}
Q= AqA^t+\bar{\lambda} BB^t,
\end{align} 
where $A,B$ and $\bar{\lambda}$ are fixed by the data set and $q$ is a $m\times m$ positive definite matrix, and $I_{d}\leq Q\leq 2^{2s} I_d$ which is equivalent to $I_{m}\leq q\leq 2^{2s} I_m$ .
\begin{align*}
p(x|Q)=\frac{1}{\sqrt{\det (2\pi Q)}}\exp\left(-\frac{1}{2}x^t Q x\right).
\end{align*}
The induced Fisher metric is simply given by
\begin{align*}
ds^2 =\frac{1}{2}\text{Tr}\left(Q(q)^{-1}dQ(q) Q^{-1}(q)dQ(q)\right)=\frac{1}{2}\text{Tr}\left(q^{-1}dq q^{-1}dq\right),
\end{align*}
where we used the map $q\mapsto Q(q)$ from Eq.~\eqref{Eq: Q} to get to the last result (formally this is called a pullback).
Note that this is exactly the same as the Fisher metric in the space of Gaussian distributions in dimension $m$, that the specific details of the subspace $V_A$ (or equivalently $V_B$) do not enter in its description, and neither does $\bar{\lambda}$. Moreover, it can be shown that the Ricci scalar~\cite{dol:per:19} for this metric is constant and equal to
\begin{align*}
R=-\frac{(m+2)m(m-1)}{4}.
\end{align*}

The Riemannian volume element in the space $\mathscr{P}_m=\{q\in \text{Mat}_{m\times m}(\mathbb{R}): q^t=q,\; q>0\}$, equipped with the Fisher metric $g=(1/2)\tr\left(q^{-1}dq q^{-1} dq\right)$ is given by (see Ref.~\cite{ter:12}, where they take a Riemannian metric which differs by a constant conformal factor $g'=2g$)
\begin{align*}
dV_g(q)= 2^{-\frac{m}{2}}\det (q)^{-\frac{(m+1)}{2}}\prod_{1\leq i\leq j\leq m}dq_{ij},
\end{align*}
where $q=[q_{ij}]_{1\leq i\leq j\leq m}$. We wish to evaluate the volume of the compact subspace $M(s)=\{q\in \mathscr{P}_m:I_{m}\leq q\leq 2^{2s} I_m \}$ with respect to this measure
\begin{align*}
\int_{M(s)}dV_g=2^{-\frac{m}{2}}\int_{M(s)}\det (q)^{-\frac{(m+1)}{2}}\prod_{1\leq i< j\leq m}dq_{ij}.
\end{align*}
Observe that the action of the group $\text{O}(m)$ on $M(s)$ by $q\mapsto K q K^t$, for $K\in \text{O}(m)$, preserves $M(s)$. One can show then that, see Ref.~\cite{ter:12}, using a parametrization $q=K a K^t$, where $a\in A=\{a\in \mathscr{P}_m: a=\text{diag}(a_1,...,a_m)\}$ and $K\in \text{O}(m)$, that 
\begin{align*}
\int_{M_s}dV_g=2^{-3\frac{m}{2}}\frac{1}{m!}\text{vol}\left(\text{O}(m)\right) \int_{[1,2^{2s}]^{m}} \prod_{j=1}^{m}a_j^{-\frac{(m-1)}{2}}\prod_{1\leq i<j\leq m}|a_i-a_j| \prod_{i=1}^{m} da_i,
\end{align*}
and $\text{vol}\left(\text{O}(m)\right)$ is volume of the orthogonal group given by
\begin{align*}
\text{vol}\left(\text{O}(m)\right)=\frac{2^{m}\pi^{\frac{m(m+1)}{4}}}{\prod_{j=1}^{m}\Gamma\left(\frac{j}{2}\right)},
\end{align*}
where $\Gamma$ is the Gamma function. To compute the volume of $M(s)$ it remains to compute the integral
\begin{align*}
I(s)=\int_{[1,2^{2s}]^{m}} \prod_{j=1}^{m}a_j^{-\frac{(m-1)}{2}}\prod_{1\leq i< j\leq m}|a_i-a_j| \prod_{i=1}^{m} da_i.
\end{align*}
A more convenient coordinate choice is provided by $a_i=e^{r_i}$ as done in Ref.~\cite{sai:bom:ber:man:17}, where now $0\leq r_i\leq \log(2^{s})$ or equivalently $0\leq r_i\leq s\log(2)$, and by the change of variables formula we get,
\begin{align*}
I(s)=8^{\frac{m(m-1)}{4}}\int_{[0,s\log(2)]^{m}}\prod_{1\leq i < j\leq m}\sinh\left(\frac{|r_i-r_j|}{2}\right)\prod_{i=1}^{m}dr_i. 
\end{align*}
Additionally, performing the change of variables $u_i=r_i/(s\log(2))$, we get
\begin{align*}
I(s)=s^{m}\left(\log(2)\right)^m 8^{\frac{m(m-1)}{4}}\int_{[0,1]^m}\prod_{1\leq i < j\leq m}\sinh\left(s\log(2)|u_i-u_j|\right)\prod_{i=1}^{m}du_i.
\end{align*} 
For large $s$, we can approximate the hyperbolic sine by the exponential  of the argument divided by two,
\begin{align*}
I(s)&\sim s^{m}\left(\log(2)\right)^m 8^{\frac{m(m-1)}{4}} 2^{-\frac{m(m-1)}{2}}\int_{[0,1]^m}\prod_{1\leq i < j\leq m}\exp\left(s\log(2)|u_i-u_j|\right)\prod_{i=1}^{m}du_i\\
&=s^{m}\left(\log(2)\right)^m 2^{\frac{m(m-1)}{4}}\int_{[0,1]^m}\prod_{1\leq i < j\leq m}\exp\left(s\log(2)|u_i-u_j|\right)\prod_{i=1}^{m}du_i\\
&=s^{m}\left(\log(2)\right)^m 2^{\frac{m(m-1)}{4}}\int_{[0,1]^m}\exp\left(s\log(2)\sum_{1\leq i < j\leq m}|u_i-u_j|\right)\prod_{i=1}^{m}du_i.
\end{align*}
The latter expression corresponds to taking the expectation value on $m$ random variables $\{u_i\}_{i=1}^{m}$ uniformly distributed on $[0,1]$. 
An upper bound for the integral is obtained by noting
\begin{align*}
\exp\left(s\log(2)\sum_{1\leq i < j\leq m}|u_i-u_j|\right)\leq \exp\left(s\log(2)\sum_{1\leq i < j\leq m}1\right)=\exp\left(s \log(2) \frac{m(m-1)}{2}\right), 
\end{align*}
for $u_i\in [0,1]$, $i=1,...,m$. This yields,
\begin{align*}
I(s)\leq  s^{m}\left(\log(2)\right)^m 2^{\frac{m(m-1)}{4}}\exp\left(s\log(2) \frac{m(m-1)}{2}\right).
\end{align*}
Since the exponential is a convex function, Jensen's inequality yields,
\begin{align*}
\exp\left(s\log(2)\int_{[0,1]^m}\sum_{1\leq i < j\leq m}|u_i-u_j|\prod_{i=1}^{m}du_i\right)\leq \int_{[0,1]^m}\exp\left(s\log(2)\sum_{1\leq i < j\leq m}|u_i-u_j|\right)\prod_{i=1}^{m}du_i,
\end{align*}
using the result $\int_{0}^{1}\int_{0}^{1}|u_1-u_1| du_1du_2=1/3$, we get
\begin{align*}
\exp\left(s\log(2)\frac{m(m-1)}{6}\right)\leq \int_{[0,1]^m}\exp\left(s\log(2)\sum_{1\leq i < j\leq m}|u_i-u_j|\right)\prod_{i=1}^{m}du_i.
\end{align*}
So that
\begin{align*}
s^{m}\left(\log(2)\right)^m 2^{\frac{m(m-1)}{4}}\exp\left(s\log(2) \frac{m(m-1)}{6}\right)\leq  I(s)\leq s^{m}\left(\log(2)\right)^m 2^{\frac{m(m-1)}{4}}\exp\left(s\log(2) \frac{m(m-1)}{2}\right),
\end{align*}
which implies
\begin{align}\label{eq:Isbound}
s\log(2) \frac{m(m-1)}{6}\leq \log I(s)- \left( m\log\left(s\log(2)\right) +\frac{m(m-1)}{4}\log(2)\right)\leq s\log(2) \frac{m(m-1)}{2}.
\end{align}
Recalling, 
\begin{align*}
\text{vol}_g\left(M(s)\right)=2^{-3\frac{m}{2}}\frac{1}{m!}\text{vol}\left(\text{O}(m)\right)I(s),
\end{align*}
we can write,
\begin{align*}
&\log\text{vol}_g\left(M(s)\right)=-\frac{3}{2}m -\log(m!) +\log\text{vol}\left(\text{O}(m)\right) +\log I(s)\\
&=-\frac{3}{2}m -\log(m!) +\log\left(\frac{2^{m}\pi^{\frac{m(m+1)}{4}}}{\prod_{j=1}^{m}\Gamma\left(\frac{j}{2}\right)}\right) +\log I(s)\\
&=-\frac{3}{2}m -\log(m!) +m\log(2)+\frac{m(m+1)}{4}\log(\pi)-\sum_{j=1}^{m}\log \Gamma\left(\frac{j}{2}\right) +\log I(s).
\end{align*}

Additionally, we have that
\begin{align*}
\sum_{j=1}^{m}\log \Gamma\left(\frac{j}{2}\right)=\log \left(\frac{\pi^{1/4} A^{3/2} G\left(\frac{m}{2}-\frac{1}{4} (-1)^{m+1}+\frac{3}{4}\right)}{{2}^{1/24} e^{1/8}}\right)+\log \left(G\left(\left\lfloor \frac{m}{2}\right\rfloor +1\right)\right),
\end{align*}
where $A$ is the Glaisher constant and $G$ is the Barnes $G$-function.
As a consequence,
\begin{align*}\label{eq:logvol}
\log\text{vol}_g\left(M(s)\right)=&-\frac{3}{2}m -\log(m!) +m\log(2)+\frac{m(m+1)}{4}\log(\pi)\\
&-\log \left(\frac{\pi^{1/4} A^{3/2} G\left(\frac{m}{2}-\frac{1}{4} (-1)^{m+1}+\frac{3}{4}\right)}{{2}^{1/24} e^{1/8}}\right)-\log \left(G\left(\left\lfloor \frac{m}{2}\right\rfloor +1\right)\right)
 +\log I(s).
\end{align*}
The bound for $\log I(s)$ is given in inequality~\eqref{eq:Isbound}. 
 And so, by Theorem~\ref{thm:regret}, and by taking $\theta_0$ to be $\hat Q$, we have that
 \begin{equation}\label{eq:gaussgc2}
L^*(x^N)=- \log p(x^N|\hat{Q})+\frac{m(m+1)}{4}\log\left(\frac{N}{2\pi}\right)+\log \textnormal{vol}_g(M(s))+\frac{(m+2)m(m-1)}{24N}.
 \end{equation} 
Observe that 
$$\log\left(\frac{\sqrt{\det(g_{\hat Q})}}{\sqrt{\det(I(x^N,\hat Q))}}\right)=0,$$
as can be noted by using the entries of the inverse of covariance matrix as coordinates in the manifold and noting that
\begin{align*}
-\frac{\partial^2 \log p(x|Q)}{\partial Q_{ij}^{-1}\partial Q_{kl}^{-1}}=\frac{1}{2}	Q_{ki}Q_{jl}, 1\leq i,j,k,l\leq m,
\end{align*}
which implies, for all $i,j,k,l$, that
\begin{align}
I_{ij,kl}(x^N,\hat{Q})=\frac{1}{N}\frac{\partial^2 \log p(x^N|Q)}{\partial Q_{ij}^{-1}\partial Q_{kl}^{-1}}\Big|_{Q=\hat{Q}}=E\left[-\frac{\partial^2 \log p(x|Q)}{\partial Q_{ij}^{-1}\partial Q_{kl}^{-1}}\right]\Big|_{Q=\hat{Q}}=g_{ij,kl}(\hat{Q}),
\end{align}
where $g_{ij,kl}(\hat{Q})$ denote the components of the Fisher metric at this point and in theses coordinates (note that, by the cyclic property of the trace, $g(Q)=(1/2)\text{Tr}\left[(Q^{-1}dQ)^2\right]=(1/2)\text{Tr}\left[(QdQ^{-1})^2\right]$).
Thus, for optimal PCA dimensional reduction, one takes the dimension $m$ that minimizes Eq.~\eqref{eq:gaussgc2} and takes the upper bound of $\log(I(s))$ as given by  inequality~\eqref{eq:Isbound}.

\section{Conclusions and outlook}\label{sec:conc}
In this paper, we derived an asymptotic formula for the posterior according to Jeffrey's prior, by extending Laplace's method to manifolds, which we called geometric complexity (see Theorem~\ref{thm:2} and compare it with Eq.~\eqref{eq:geomcomp}). Then, we provided the minmax regret for general statistical manifolds by introducing the notion of locally sound smooth families of probability distributions, which builds on Clarke and Barron's results for bounded open sets in $\mathbb{R}^n$. Finally, we gave an explicit formula of the geometric complexity for families of Gaussian distributions with zero-mean, and varying covariance, and apply this formula to optimal dimensional reduction in PCA.

Future work includes finding more expressions of the geometric complexity for other families of probability distributions. Another interesting area of research is to understand the higher-order corrections to the Geometric complexity, as they might be relevant for high dimensional data.

\appendices
\section{Proof of Theorem~\ref{th: Laplace Riemannian orientable closed}}
We begin by recalling the analogous result valid in $\mathbb{R}^n$.
\begin{theorem} (Laplace's method) Let $f\in C^2(\mathbb{R}^n)$, with $\int_{\mathbb{R}^n} e^{-f(x)}dx<\infty$, such that there exists a unique $x_0$ with
\begin{align*}
df(x_0)=0 \text{ and } \hess(f)(x_0)<0,
\end{align*}
i.e., $x_0$ is the unique global maximum of $f$. Suppose additionally that for every $x\in \mathbb{R}^n-\{x_0\}$, we have that $f(x)<f(x_0)$, i.e, $f(x_0)$ is really the maximum value $f$ can have. Then
\begin{align*}
\lim_{N\to \infty}\frac{\int_{\mathbb{R}^n} e^{Nf(x)}dx}{e^{Mf(x_0)}\sqrt{\det\left[2\pi (-N\hess(f)(x_0))^{-1}\right]}}=1.
\end{align*}
\end{theorem}

\begin{remark}\label{rmk:Laplace} Another useful formulation of the above theorem found recurrently in the literature is given by
\begin{align*}
\int_{\mathbb{R}^n}h(x) e^{Nf(x)} dx\sim h(x_0) e^{Nf(x_0)}\sqrt{\det\left[2\pi (-N\hess(f)(x_0))^{-1}\right]}, \text{ as } N\to\infty.
\end{align*}
for a function $h$.
\end{remark}

Now let $(M,g)$ be a compact closed oriented Riemannian manifold of dimension $n$ and let $dV_g=\sqrt{\det(g)(x)}dx^1\wedge...\wedge dx^n$ be associated Riemannian volume form written in local coordinates $(x^1,...,x^n)$. We wish to generalize Laplace's method to integrals of the form
\begin{align*}
\int_{M}e^{Nf}dV_g,
\end{align*}
for large positive $N$ and $f$ being a smooth function with non-degenerate maximum at $p_0$. Recall that, at $p_0$, there is a well-defined non-degenerate bilinear form $\hess_{p_{0}}(f):T_{p_{0}}M\times T_{p_{0}}M\to\mathbb{R}$ defined by
\begin{align*}
\hess_{p_{0}}(f)(X,Y)=\widetilde{X}\cdot (\widetilde{Y}\cdot f) (p_{0}),
\end{align*}
where $\widetilde{X}$ and $\widetilde{Y}$ are arbitrary extensions of $X,Y\in T_{p_{0}}M$ to vector fields in an open neighbourhood of $p_{0}$. 

We will also need the following result.
\begin{proposition} Let $(x^1,...,x^n)$ be Riemann normal coordinates centered at some point $p$ defined in some open neighborhood $U\subset M$, then, there exists a neighborhood of $p$, $V\subset U$, such that
\begin{align*}
\sqrt{\det(g(x))}= 1- \frac{1}{6}\sum_{i,j=1}^{n}R_{ij}(0)x^{i} x^{j} +\mbox{O}(||x||^3),
\end{align*}
where $R_{ij}(0)$ are the components of the Ricci tensor with respect to the $x^i$'s.
\label{prop: Riemann normal coords}
\end{proposition}
Using Proposition~\ref{prop: Riemann normal coords}, we can now proceed to the proof of Theorem~\ref{th: Laplace Riemannian orientable closed}.
\begin{proof}[Proof of Theorem~\ref{th: Laplace Riemannian orientable closed}] 
Take $\mathfrak{A}=\{U_{k}\}_{k=1}^{K}$, $K<\infty$ (since $M$ is compact we can take a subcover if necessary so that it is finite), an open cover of $M$ associated with positively oriented charts $\varphi_k:U_k\to\mathbb{R}^n$ and let $\{f_k\}$ denote a partition of unity subordinate to $\mathfrak{A}$. Then,
\begin{align*}
\int_{M}e^{Nf}dV_g &=\int_{M}\sum_{k=1}^{K}f_ke^{Mf}dV_g\\
&=\sum_{k=1}^{K}\int_{U_k}f_k e^{Nf}dV_g\\
&=\sum_{k=1}^{K}\int_{\varphi_k(U_k)}f_k\circ\varphi_k^{-1} e^{Nf\circ\varphi_k^{-1}}(\varphi_k^{-1})^{*}dV_g.
\end{align*} 
The functions $f_k$, by definition, satisfy $1\geq f_k(p)\geq 0$ for every $p\in M$. Fix a $k\in \{1,...,K\}$. Suppose $p_0\notin U_k$. Since $M-U_k$ is a closed subset of a compact space it is compact. Therefore $f$ reaches a maximum value say $f(p_0)-\eta$ in $M-U_k$, for some $\eta>0$. Therefore,
\begin{align*}
0\leq \int_{U_k}f_k e^{Nf}dV_g &\leq \int_{M}e^{f}e^{(N-1)(f(p_0)-\eta)}dV_g\\
&\leq e^{(N-1)(f(p_0)-\eta)}\int_{M}e^{f}dV_g.
\end{align*}
If we divide both sides by $\left(\frac{2\pi}{N}\right)^{n/2}\frac{e^{Nf(p_0)}}{\sqrt{\det(\hess_{p_0}(f))}}\left[1+\frac{1}{6N}\textnormal{tr}(\hess_{p_0}(f)^{-1}R_{p_0})\right]$ and take the limit $N\to\infty$, it is then clear that this contribution will vanish and, thus, have no role. For simplicity, and without loss of generality, we assume that $p_0$ is in $U_k$ for a single $k$ only. Then, we need to focus on
\begin{align*}
\int_{U_k}f_k e^{f}dV_g=\int_{\varphi_k(U_k)}f_k\circ\varphi_k^{-1} e^{Nf\circ\varphi_k^{-1}}(\varphi_k^{-1})^{*}dV_g.
\end{align*}
We assume, without loss of generality, $\varphi_k=(x^1,...,x^n)$ to be a normal coordinate system centered $p$ and by abuse of notation denote $f\circ\varphi_k^{-1}$ by simply $f$ and $f_k\circ\varphi_k^{-1}$ by simply $f_k$. The image $\varphi_k(U_k)$ is an open set in $\mathbb{R}^n$, which we will denote $V$. We are then dealing with the integral
\begin{align*}
\int_{V}f_{k}(x)e^{Nf(x)}\sqrt{\det g(x)}dx.
\end{align*}
We can take a smaller open subset $W\subset V$, with $\varphi_k(p_0)=0\in W$, where $f_k|_{W}=1$. Notice that over $V-W$, since the maximum of $f$ is reached for $0\in W$, we have, quite similarly to what we did above,
\begin{align*}
\int_{V-W}f_{k}(x)e^{Nf(x)}\sqrt{\det g(x)}dx &\leq \int_{V}f_{k}(x)e^{f(x)}e^{(N-1)(f(0)-\eta)}\sqrt{\det g(x)}dx \\
&=e^{(N-1)(f(0)-\eta)} \int_{V}f_{k}(x)e^{f(x)}\sqrt{\det g(x)}dx,
\end{align*}  
where $\eta>0$ exists since $f(p_0)>f(p)$ for all $p\in M$, and the inequality follows from the integral being positive. When we divide both sides by $\left(\frac{2\pi}{N}\right)^{n/2}e^{Nf(p_0)}\!\!\frac{\sqrt{\!\det(g_{p_0})}}{\sqrt{\det(\hess_{p_0}(f))}}\left[\!1\!+\!\frac{1}{6N}\textnormal{tr}(\hess_{p_0}(f)^{-1}R_{p_0})\!\right]$ it is clear that this term goes to zero in the limit $N\to\infty$. It is then enough to consider the integral
\begin{align*}
\int_{B_{\delta}(0)}e^{Nf(x)}\sqrt{\det g(x)}dx,
\end{align*}
where we have replaced $W$ by a ball $B_{\delta}(0)$ containing $\varphi_k(p_0)=0$. Next, by choosing $\delta$ sufficiently small, we can use Proposition~\ref{prop: Riemann normal coords} to write:
\begin{align*}
\sqrt{\det g(x)}=1-\frac{1}{6}\sum_{i,j=1}^{n}R_{ij}(0)x^{i} x^{j} +\text{O}(||x||^3).
\end{align*}
By the identifications $T_{p_0}M\cong\mathbb{R}^n$ provided by normal coordinates, this can be reformulated as
\begin{align*}
\sqrt{\det g(x)}= 1-\frac{1}{6}x^t R_{p_0}x +||x||^3g(x),
\end{align*}
where we see $R_{p_0}$ as an $n\times n$ matrix and $g(x)$ is some function with the property $g(x)\to 0$ as $x\to 0$. By compactness of $\overline{B_{\delta}(0)}$, there exists a constant $C>0$, such that
\begin{align*}
\left|\sqrt{\det g(x)}- (1-\frac{1}{6}x^t R_{p_0}x)\right|\leq C||x||^3.
\end{align*} 
We can replace $||x||^3$ on the right hand side by the absolute value of an arbitrary polynomial in the $x^i$'s whose first term is of degree $3$, let us call it $P(x)=\sum_{i_1,i_2,i_3=1}^{n} a_{i_1i_2i_3}x^{i_1}x^{i_2}x^{i_3}+...$, with an appropriate new choice for $C$.
Therefore,
\begin{align*}
\left|\int_{B_{\delta}(0)}e^{Nf(x)}\sqrt{\det g(x)}dx -\int_{B_{\delta}(0)}e^{Nf(x)}(1-\frac{1}{6}x^t R_{p_0}x)dx\right|&\leq \int_{B_{\delta}(0)}e^{Nf(x)}\left|\sqrt{\det g(x)} -(1-\frac{1}{6}x^t R_{p_0}x)\right|dx \\
&\leq C\int_{B_{\delta}(0)} e^{Nf(x)}|P(x)|dx.
\end{align*}
Next, we let $A=-\text{Hess}_{p_0}(f)$ and perform the change of variables according to $y=\sqrt{N}A^{1/2}x=:F(x)$. Notice that $F$, as defined, defines a diffeomorphism of open sets in $\mathbb{R}^n$, where we see $A$ as a linear endomorphism of $\mathbb{R}^n$ using the orthogonal normal coordinates. It is clear that as $N\to\infty$ the image under $F$ of $B_{\delta}$ becomes $\mathbb{R}^n$. We then have
\begin{align*}
C\int_{B_{\delta}(0)} e^{Nf(x)}|P(x)| dx= C\det(N^{-1/2}A^{-1/2})\int_{F(B_{\delta}(0))} e^{Nf\circ F^{-1}(y)}|P\circ F^{-1}(y)|dy.
\end{align*}
Now $Nf\circ F^{-1}(y)= Nf(0)-1/2 ||y||^2+ \mbox{O}(N^{-1/2}||y||^3)$. As $N$ grows larger, all we need to do is the integral over $\mathbb{R}^n$ of $ e^{Nf\circ F^{-1}}|P\circ F^{-1}|$, which by Laplace's approximation in $\mathbb{R}^n$, see Remark~\ref{rmk:Laplace}, is proportional to evaluating $|P|$ at $0$, which yields zero. Therefore,
\begin{align*}
&\lim_{N\to\infty}\int_{B_{\delta}(0)}e^{Nf(x)}\sqrt{\det g(x)}dx=\lim_{N\to\infty}\int_{B_{\delta}(0)}e^{Nf(x)}(1-\frac{1}{6}x^t R_{p_0}x)dx.
\end{align*}
Moreover, for finite $N$,
\begin{align*}
\int_{B_{\delta}(0)}e^{Nf(x)}(1-\frac{1}{6}x^t R_{p_0}x)dx=&\det(N^{-1/2}A^{-1/2})e^{Nf(0)} \\
&\times\int_{F(B_{\delta}(0))}e^{N f\circ F^{-1}(y)}(1-\frac{1}{6N}y^tA^{-1/2}R_{p_0}A^{-1/2}y)dy.
\end{align*}
In the large $N$ limit, we just need to evaluate the Gaussian integral, yielding
\begin{align*}
\left(\frac{2\pi}{N}\right)^{n/2}\frac{e^{Nf(0)}}{\sqrt{\det(\hess_{p_0}(f))}}\left[1+\frac{1}{6N}\tr(\hess_{p_0}(f)^{-1}R_{p_0})\right].
\end{align*}
We then get
\begin{align*}
\lim_{N\to\infty}\frac{\int_{M} e^{Nf}dV_g }{\left(\frac{2\pi}{N}\right)^{n/2}\frac{e^{Nf(p_0)}}{\sqrt{\det(\hess_{p_0}(f))}}\left[1+\frac{1}{6N}\tr(\hess_{p_0}(f)^{-1}R_{p_0})\right]}=1.
\end{align*}
Note that the identification of $-\hess_{p_0}(f)$ as a linear map implies the use of the metric $g_{p_0}$ at $T_{p_0}M$, which in the orthogonal normal coordinates is just the identity matrix. Therefore, the invariant form of $\sqrt{\det(\hess_{p_0}(f))}$ is $\sqrt{\det(\hess_{p_0}(f))}/\sqrt{\det(g_{p_0})}$, where now $\hess_{p_0}(f)$ and $g_{p_0}$ are understood as the bilinear forms $\hess_{p_0}(f)$ and $g_{p_0}$ expressed as matrices in arbitrary, but of course the same, coordinates. This yields the final result:
\begin{align*}
&\lim_{N\to\infty}\frac{\int_{M} e^{Nf}dV_g }{\left(\frac{2\pi}{N}\right)^{n/2}e^{Nf(p_0)}\frac{\sqrt{\det(g_{p_0})}}{\sqrt{\det(\hess_{p_0}(f))}}\left[1+\frac{1}{6N}\tr(\hess_{p_0}(f)^{-1}R_{p_0})\right]}=1.
\end{align*}
\end{proof}
\begin{remark} One can extend the results to the paracompact case, i.e., $(M,g)$ an arbitrary oriented Riemannian manifold without boundary, with the additional assumptions that $\int_{M}e^{Nf}dV_g<\infty$ for some finite $N$ and that $f(p_0)$ is the maximum value $f$ attains $M$ (assumptions which are immediate for compact $M$).
\end{remark}

\section*{Acknowledgment}

BM and PM thank the support  from SQIG -- Security and Quantum Information Group. BM, PM and AC  thanks the Funda\c{c}\~ao para a Ci\^{e}ncia e a Tecnologia (FCT) project UID/EEA/50008/2020, and European funds, namely H2020 project SPARTA. BM, PM and AC acknowledge PREDICT PTDC/CCI-CIF/29877/2017 funded by FCT. We also acknowledge J. Mour\~ao e J. P. Nunes for valuable discussions concerning the Laplace formula in the context of manifolds. 
We also acknowledge discussions with colleagues from the Electric Engineering department concerning the applications of manifolds to Information Theory.

\ifCLASSOPTIONcaptionsoff
  \newpage
\fi

\bibliographystyle{unsrt}
\bibliography{bib}

%
%

\end{document}